\documentclass[a4paper,oneside]{amsart}

\usepackage[usenames,dvipsnames,table]{xcolor}
\usepackage{xcolor}

\usepackage{amsmath, amsthm, amssymb, mathrsfs, calc}
\usepackage{mathtools}
\usepackage{url}
\usepackage{mleftright}
\usepackage{framed} \colorlet{shadecolor}{blue!20}
\usepackage{microtype}
\usepackage[small,basic]{complexity}
\usepackage{tikz}
\usetikzlibrary{arrows, positioning}
\usepackage{aliascnt}
\usepackage{hyperref}
\usepackage[colorinlistoftodos,prependcaption,textsize=tiny,obeyDraft]{todonotes} %%% ObeyFinal deletes todonotes only if final is a global option

\newtheorem{theorem}{Theorem}[section]

\newaliascnt{lem}{theorem}
\newtheorem{lem}[lem]{Lemma}
\aliascntresetthe{lem}

\newaliascnt{cor}{theorem}
\newtheorem{cor}[cor]{Corollary}
\aliascntresetthe{cor}

\newaliascnt{prop}{theorem}

\aliascntresetthe{prop}

\newaliascnt{claim}{theorem}

\aliascntresetthe{claim}

\theoremstyle{definition}

\newaliascnt{example}{theorem}
\newtheorem{example}[example]{Example}
\aliascntresetthe{example}

\newaliascnt{defin}{theorem}

\aliascntresetthe{defin}

\newaliascnt{rem}{theorem}

\aliascntresetthe{rem}

\newaliascnt{conj}{theorem}

\aliascntresetthe{conj}

\newaliascnt{ass}{theorem}

\aliascntresetthe{ass}

\numberwithin{equation}{section}

\renewcommand{\mathcal}[1]{\mathscr{#1}}

\renewcommand{\R} {\mathbb{R}} \newcommand{\Z} {\mathbb{Z}}  \newcommand{\N} {\mathbb{N}}    \newcommand{\Q} {\mathbb{Q}}

\newclass{\PTIME}{PTIME}

\makeatletter
\DeclareRobustCommand\bigop[1]{%
  \mathop{\vphantom{\sum}\mathpalette\bigop@{#1}}\slimits@
}
\newcommand{\bigop@}[2]{%
  \vcenter{%
    \sbox\z@{$#1\sum$}%
    \hbox{\resizebox{\ifx#1\displaystyle.9\fi\dimexpr\ht\z@+\dp\z@}{!}{$\m@th#2$}}%
  }%
}
\makeatother

\providecommand*{\eu}%
{\ensuremath{\mathrm{e}}}
\providecommand*{\iu}%
{\ensuremath{\mathrm{i}}}

\usepackage{newpxtext}
\usepackage{newpxmath} %[smallerops]
\renewcommand{\mathbb}[1]{\varmathbb{#1}}

\usepackage[T1]{fontenc}
\usepackage[utf8]{inputenc}

\title{On the Skolem Problem and Prime Powers} %TODO Please add

\author{George Kenison}
\address{George Kenison, Department of Computer Science, University of Oxford, Oxford, UK.}
\email{george.kenison@cs.ox.ac.uk}
\author{Richard Lipton}
\address{Richard Lipton, Georgia Institute of Technology, Atlanta, USA.}
\email{richard.lipton@cc.gatech.edu}

\author{Jo\"{e}l Ouaknine}
\thanks{The third author is supported by ERC grant AVS-ISS (648701) and DFG grant 389792660 as part of TRR 248 (see \url{https://perspicuous-computing.science}).  Also affiliated to Department of Computer Science, Oxford University, Oxford, UK.}
\address{Jo\"{e}l Ouaknine, Max Planck Institute for Software Systems, Saarbr\"{u}cken, Germany.}
\email{joel@mpi-sws.org}

\author{James Worrell}
\thanks{The fourth author is supported by EPSRC Fellowship EP/N008197/1.}
\address{James Worrell, Department of Computer Science, University of Oxford, Oxford, UK.}
\email{jbw@cs.ox.ac.uk}

\thanks{Accepted for publication in the proceedings of the International Symposium on Symbolic and Algebraic Computation, ISSAC 2020.}

\begin{document}

\begin{abstract}
The Skolem Problem asks, given a linear recurrence sequence \((u_n)\),
whether there exists \(n\in\N\)  such that \(u_n=0\).
In this paper we consider the following specialisation of the problem: given in addition $c\in\mathbb{N}$, determine whether there
  exists \(n\in\N\)  of the form \(n=lp^k\),  with \(k,l\leq c\)
  and \(p\) any prime number, such that \(u_n=0\).
\end{abstract}

\maketitle

\section{Introduction}
A sequence \((u_n)_{n=0}^\infty\) of real algebraic numbers is called a \textit{linear recurrence sequence} if its terms satisfy a recurrence relation \(u_n = a_1 u_{n-1} + a_2 u_{n-2} + \cdots + a_\ell u_{n-\ell}\), with fixed real algebraic constants \(a_1,\ldots, a_\ell\) such that \(a_\ell\neq 0\).   Such a recurrence is said to have order $\ell$ and
a sequence \((u_n)\) satisfying the recurrence is wholly determined by the initial values \(u_0,\ldots, u_{\ell-1}\).  
The study of linear recurrence sequences is motivated by a wide range of phenomena, in areas such as analysis of algorithms, and biological and economic modelling.  %
Natural decision problems for linear recurrence sequences include:
whether all the terms in a sequence are positive, whether the
terms of the sequence are eventually positive, and whether the
sequence contains a zero.  The latter, commonly known as the Skolem Problem~\cite{everest2003recurrence,halava2005skolem}, is the main object of study in the current paper.

Let \((u_n)\) be a linear recurrence sequence. A remarkable result of
Skolem, Mahler, and Lech states that the set \(\{n\in\N : u_n=0\}\) is the union of a finite set together with a finite number of (infinite) arithmetic progressions.  The original result, proved by Skolem~\cite{skolem1934verfahren} for the field of rational numbers, was subsequently extended to the field of algebraic numbers by Mahler~\cite{mahler1935taylor, mahler1956taylor}, and then further extended to any field of characteristic \(0\) by Lech~\cite{lech1952recurring}.
All known proofs of the Skolem-Mahler-Lech Theorem (as it is now
known) employ techniques from \(p\)-adic analysis.  These proofs are non-constructive and the decidability of the Skolem Problem remains open.  Berstel and Mignotte, however, gave an effective method to obtain all of the arithmetic progressions in the statement of the theorem~\cite{berstel1976deux}.

For fields of positive characteristic, the conclusion of the Skolem-Mahler-Lech Theorem does not hold. Indeed, Lech~\cite{lech1952recurring}  gave the following illustrative example. Let \(K=\mathbb{F}_p(t)\) and consider the sequence with terms \(u_n = (1+t)^n -t^n -1\). Then \((u_n)\) satisfies a linear recurrence over \(K\), but \(u_n=0\) if, and only if, \(n=p^k\).  Nevertheless, Derksen~\cite{derksen2007skolem} established an analogue of the Skolem-Mahler-Lech Theorem for fields of positive characteristic, namely he proved that  the set of zeroes in a field of characteristic \(p\) is a \(p\)-automatic set.  The proof of Derksen was moreover effective, allowing to construct for a given sequence the automaton representing the set 
of its zeros.

Returning to the characteristic-zero setting, progress on the decidability of the Skolem Problem has been
made by restricting the problem to linear recurrence sequences of
low order. Decidability of the Skolem Problem for sequences of order at most \(2\) is straightforward and the results are considered folklore.  Breakthrough work by Mignotte, Shorey, and Tijdeman~\cite{mignotte1984distance}, and, independently, Vereshchagin~\cite{vereshchagin1985occurence},  showed decidability of the Skolem Problem for linear recurrence sequences of order \(3\) and \(4\).
Techniques from \(p\)-adic
analysis and algebraic number theory are employed in both~\cite{mignotte1984distance} and~\cite{vereshchagin1985occurence}. Both papers moreover make critical use of Baker's theorem for
linear forms in logarithms of algebraic numbers.  
The approach via Baker's Theorem taken in the above papers does not appear to extend easily to recurrences of higher order.  In particular, decidability of Skolem's Problem remains open
for recurrences of order \(5\).  
However, the recent resurgence of research activity concerning the
decidability of various sub-cases of the Skolem Problem and related
questions (see the survey~\cite{ouaknine2015linear}) gives an
indication of its fundamental importance to the field.

In this paper we pursue an alternative approach to restricting the order of the recurrence as a means of obtaining decidable specialisations of Skolem's Problem.  
We consider general recurrences, but ask to decide the existence of zeros of certain prescribed forms.
For example, we ask whether
one can show 
decidability of the Skolem Problem when we consider only those
\(n\in\N\) that are prime powers. Our first basic result---which we will
generalise in various ways in the rest of the paper---is the following,
which applies to a class of \textit{simple} linear
recurrence sequences (i.e., those sequences without repeated characteristic roots):

\begin{theorem} \label{thm: FirstCase1}
	Suppose that each term in a linear recurrence sequence
        \((u_n)\) can be written as an algebraic exponential
        polynomial \(u_n= A_1\lambda_1^n+\cdots +A_m\lambda_m^n\) with
        \(A_1,\ldots, A_m\in\mathbb{Z}\) and \(\lambda_1,\ldots,
        \lambda_m\) distinct algebraic integers. Fix \(c\in\N\). Then one can decide whether there exists \(n\in\{p^k : p \text{ prime}, k\le c\}\) such that \(u_n=0\).
\end{theorem}

In general, a simple linear recurrence sequence \((u_n)\) has the property
that each of its terms is given by an algebraic exponential polynomial
\(u_n = A_1\lambda_1^n + \cdots + A_m\lambda_m^n\) with \(A_1,\ldots,
A_m\in\mathfrak{O}\) algebraic integers in a number field \(K\).  In
Theorem~\ref{thm: FirstCase1} we assumed that \(A_1,\ldots,
A_m\in\Z\).  More generally, 
a linear recurrence sequence \((u_n)\) can always be written in the
form \(u_n = A_1(n)\lambda_1^n + \cdots + A_m(n)\lambda_m^n\), 
where the $A_i$ are univariate polynomials and the $\lambda_i$ are
characteristic roots of the recurrence relation. We establish decidability results for linear recurrence sequences \((u_n)\) in this general setting. We consider the case of rational polynomial coefficients in Section 3; that is, \(A_1,\ldots, A_m\in\mathbb{Z}[x]\) and, more generally, algebraic polynomial coefficients in Section 5.  We outline two generalisations of Theorem~\ref{thm: FirstCase1} below.

%There are three aspects of Theorem~\ref{thm: FirstCase1} that we generalise in this note:
%
%

First, assume that the linear recurrence sequence \((u_n)\) satisfies \(u_n = A_1(n)\lambda_1^n + \cdots + A_m(n)\lambda_m^n\) such that \(A_1,\ldots, A_m\in\mathbb{Z}[x]\).  The next result follows as a corollary to Theorem~\ref{thm: SecondCase}. In the proof of Theorem~\ref{thm: SecondCase} we introduce and analyse an associated simple linear recurrence \((v_n)\) with terms \(v_n = A_1(0)\lambda_1^n + \cdots + A_m(0)\lambda_m^n\).
\begin{theorem} \label{thm: introthm} 
	Let \((u_n)\) be a recurrence sequence with rational polynomial coefficients and \((v_n)\) the associated simple recurrence. Fix \(c\in\N\).
	If \(v_1\neq 0\) then one can
        decide whether  there exists \(n\in\{p^k : p \text{ prime}, k\le c\}\) such that \(u_n=0\).
\end{theorem}

%
%%
%%%
%%%%
%%%%%
%%%%%%

%Second, we give a slight improvement over the decidability for prime powers in the sense that we consider natural numbers of the form \(lp^k\).  Let us assume that \((u_n)\) satisfies \(u_n = A_1(n)\lambda_1^n + \cdots + A_m(n)\lambda_m^n\) such that \(A_1,\ldots, A_m\in\mathbb{Z}[x]\).
%\begin{theorem} \label{thm: introthm2}
%	Let \((u_n)_{n=0}^\infty\) be a linear recurrence
  %      sequence. Fix \(c\in\N\).  Suppose that \(l\le c\) and
    %    \(v_l\neq 0\). Then one can decide whether  there exists an \(n\in\{lp^k : p \text{ prime}, k\le c\}\) such that \(u_n=0\).
%\end{theorem}
%Theorem~\ref{thm: introthm2} follows from Theorem~\ref{thm: SecondCase}.  In Theorem~\ref{thm: SecondCase} we prove decidability over a union of sets of the form \(\{lp^k : p \text{ prime}, k\le c\}\).  This union is taken over \(l\le c\) for which \(v_l\neq 0\).

%
%%
%%%
%%%%
%%%%%

Now suppose that the terms of \((u_n)\) are given by
\(u_n = A_1(n)\lambda_1^n
+ \cdots + A_m(n)\lambda_m^n\) where the coefficients \(A_1,\ldots,
A_m\in\mathfrak{O}[x]\) are univariate polynomial with \(\mathfrak{O}\) the ring of integers
of a finite Galois extension \(K\) over \(\mathbb{Q}\). 
As before, let \((v_n)\) be the associated simple recurrence.
 To each rational prime \(p\) we associate a constant \(f(p)\) (the \textit{inertial degree} of \(p\Z\) in \(K\)). The next result follows as a corollary to Theorem~\ref{thm: GeneralResult}.
\begin{theorem} \label{thm: introthm3} 
	Suppose that \((u_n)\) is a recurrence sequence with algebraic polynomial coefficients and \((v_n)\) the associated linear
        recurrence as above. Fix
        \(c\in\N\).  If \(v_1\neq 0\) then one can decide whether there exists \(n\in\{p^{kf(p)} : p \text{ prime}, k\le c\}\) such that \(u_n=0\).
\end{theorem}

We motivate our decidability results with a discussion of the
decidability of the Skolem Problem for linear recurrence sequences of
order \(5\). The authors of~\cite{halava2005skolem} claim to prove
that the Skolem Problem is decidable for integer linear recurrence
sequences of order \(5\);  however, as pointed out in
\cite{ouaknine2012decision}, there is a gap in the argument.  The
critical case for which the decidability of the Skolem Problem is open is that of a recurrence sequence of order \(5\) whose characteristic polynomial has five distinct roots: four distinct roots \(\lambda_1,\overline{\lambda_1},\lambda_2,\overline{\lambda_2}\in\mathbb{C}\) such that \(|\lambda_1|=|\lambda_2|\), and a fifth root \(\rho\in\R\) of strictly smaller magnitude.  In this case the terms of such a recurrence sequence \((u_n)\) are given by \(u_n = a\mleft(\lambda_1^n + \overline{\lambda_1^n}\mright) + b\mleft(\lambda_2^n + \overline{\lambda_2^n}\mright) + c\rho^n\). Here \(a,b,c\in\R\) are algebraic numbers.  If \(|a|\) and \(|b|\) are not equal then there is no known general procedure to determine \(\{n\in\N : u_n=0\}\).

Next we consider an example of a linear recurrence sequence from the aforementioned critical case.  We motivate the results herein and also illustrate the techniques used in this paper by demonstrating that the sequence does not vanishes at any prime index.

%
 %%
 %%%
 %%%%
 
\begin{example} \label{ex: example}
For this example set \(\lambda_1 = 39+52\iu\), \(\lambda_2 = -60+25\iu\) and \(\rho=1\).  (Our choices of Pythagorean triples \((39,52,65)\) and \((25,60,65)\) ensure that \(|\lambda_1|=|\lambda_2|=65\).) 
Let \((v_n)\) be the linear recurrence sequence whose terms satisfy 
  \begin{equation*}v_n = \lambda_1^n + \overline{\lambda_1^n} + 3\mleft(\lambda_2^n + \overline{\lambda_2^n}\mright) + \rho^n.
  \end{equation*}
 There are no rational primes \(p\in\N\) for which \(v_p=0\).

We omit many technical definitions and details in the following presentation (for such details we refer the reader to the preliminariy material in the next section).
\begin{proof}[Proof of Example~\ref{ex: example}]
    Let \(K\) be the \textit{splitting field} of the minimal polynomial (over \(\Q\)) associated to \((v_n)\).  %Over \(K\) the minimal polynomial can be written as a product of distinct linear factors. 
  We find that \(K=\mathbb{Q}(\lambda_1, \overline{\lambda_1}, \lambda_2,\overline{\lambda_2}, 1) \cong \mathbb{Q}(i)\).  The \textit{dimension} \(d\) of the field \(K\) as a vector space over \(\mathbb{Q}\) is  \(2\). 
There is a computable constant \(N\in\N\) depending only on \(v_1\) and the field \(K\) introduced in the preliminaries---the norm of the principal ideal generated by \(v_1\)---with the following property.
  Suppose that \(p\in\N\) is a rational prime.  Then, by Corollary~\ref{cor: freshman} and Lemma~\ref{lem: pbound}, \(v_p=0\) only if \(p\vert N\). 

Assume that \(v_p = 0\) for some prime \(p\in\N\).
We calculate \(v_1=-281\), which we use to determine \(N\).
Here \(N=|v_1|^d = 281^2\).  
Thus \(p\vert N = 281^2\) from our assumption. 
By happy coincidence, 281 is a rational prime and so it is sufficient to check whether \(v_p=0\) for the only possible candidate \(p=281\). 
Using Mathematica we compute \(v_{281} \approx 3.7\times 10^{509}\) (to two significant figures).  
We conclude that there does not exist a rational prime \(p\in\N\) such that \(v_p=0\).
\end{proof}

%%Note that it is genuinely open whether $v_n = 0$ for some $n \in \mathbb{N}$.
  
 \end{example}

This paper is organised as follows. In Section 2, we recall preliminary terminology and background material from algebraic number theory and recurrence sequences.  In Section 3, we prove decidability results locating zeroes of recurrence sequences of the form \(u_n = A_1(n)\lambda_1^n + \cdots + A_m(n)\lambda_m^n\) with polynomial coefficients \(A_1,\ldots, A_m\in\Z[x]\) having integer coefficients.  The main result in Section 3 is Theorem~\ref{thm: SecondCase}. 
%In Section 4, we recall concepts related to the decomposition of prime ideals in a Galois extension. 
In Section 4 we prove decidability results for linear recurrence sequences with polynomial coefficients  \(A_1,\ldots, A_m\in \mathfrak{O}[x]\), where \(\mathfrak{O}\) is the ring of integers of a Galois number field.  
The main result in Section 4 is Theorem~\ref{thm: GeneralResult}.
In Section 5 we show that the problem of deciding whether a given linear recurrence sequence has a prime zero is \NP-hard.  This matches the best known
lower bound for the general Skolem Problem.

%
%%
%%%
%%%
%%%
%%%
%%%
%%%

\section{Algebraic number theory and linear recurrence sequences}
In this section we recall some basic notions concerning algebraic numbers and linear recurrences that will be used in the sequel.

A complex number \(\alpha\) is \textit{algebraic} if there exists a polynomial \(P\in\mathbb{Q}[x]\) such that \(P(\alpha)=0\).  
%We write \(\mathbb{A}\) for the set of all algebraic numbers and note that \({\mathbb{A}}\) is a subfield of \(\mathbb{C}\). 
The \textit{minimal polynomial} of \(\alpha\in{\mathbb{A}}\) is the unique monic polynomial \(\mu_\alpha\in\Q[x]\) of least degree such that \(\mu(\alpha)=0\).  The \textit{degree} of \(\alpha\), written \(\deg(\alpha)\), is the degree of its minimal polynomial. An \textit{algebraic integer} \(\alpha\) is an algebraic number whose minimal polynomial has integer coefficients.  The collection of all algebraic integers forms a ring \(\mathbb{B}\).

A \textit{number field} \(K\) is a field extension of \(\mathbb{Q}\) whose dimension as a vector space over \(\mathbb{Q}\) is finite.
We call the dimension of this vector space the \textit{degree} of the number field and use the notation \([K\colon\mathbb{Q}]\) for the degree of \(K\).
Call a number field \(K\) \textit{Galois} if it is the splitting field of some separable polynomial over $\mathbb{Q}$.
%We denote by \(\mathfrak{O}\) the ring of algebraic integers in \(K\).  
Let   \(\mathfrak{O}=\mathbb{B}\cap K\) be the ring of algebraic integers in \(K\).  
Because \(\mathbb{B}\cap \mathbb{Q}=\mathbb{Z}\), we refer to the elements of \(\Z\) as \textit{rational integers}.   
For each \(\alpha\in K\) there exists a non-zero \(q\in\mathbb{Z}\) such that \(q\alpha\in\mathfrak{O}\).

Given a number field \(K\) of degree $d$ over $\mathbb{Q}$, there are exactly \(d\) distinct monomorphisms \(\sigma_i \colon K \to \mathbb{C}\).
We define the \textit{norm} \(N_K(\alpha)\) of \(\alpha\in K\) by
    \begin{equation*}
        N_K(\alpha) = \prod_{\ell =1}^d \sigma_\ell(\alpha).
    \end{equation*}
Then $N_K(\alpha) \in \mathbb{Q}$ and furthermore $N_K(\alpha)\in\mathbb{Z}$ if $\alpha\in \mathfrak{O}$.
%It is moreover clear that \(N_{\mathbb{Q}(\alpha)}(\alpha)\) is equal (up to sign) to the constant term of the minimal polynomial \(\mu_\alpha\). 

Suppose that \(P\in\Z[x]\) is a polynomial with integer coefficients. 
The \textit{height} of \(P\) is the maximum of the absolute values of its coefficients
and write \(\| P\|\) for the bit length of the list of its coefficients encoded in binary.
It is clear that the degree of \(P\) is at most \(\|P\|\), and the height of \(P\) is at most \(2^{\|P\|}\).

There is a standard representation of an algebraic number \(\alpha\) as a tuple \((\mu_\alpha, a, b, \varepsilon)\) where \(\mu_\alpha\) is the minimal polynomial of \(\alpha\) and \(a,b,\varepsilon\in\mathbb{Q}\) with \(\varepsilon>0\) sufficiently small so that \(\alpha\) is the unique root of \(\mu_\alpha\) inside the ball of radius \(\varepsilon\) centred at \(a+b\iu\in\mathbb{C}\). 
Given a polynomial \(P\in\Z[x]\), we can compute a standard representation for each of its roots in time polynomial in \(\|P\|\).  
%
%
%
%
%%
%%%%
%%%%
%%%%%

%\subsection{Ideals of a Dedekind domain}
We recall some standard terminology and basic results about ideals in \(\mathfrak{O}\).  
The ideal \(\mathfrak{a}=a\mathfrak{O}\) generated by a single element \(a\in\mathfrak{O}\) is called \textit{principal}. 
% An \textit{ideal} \(\mathfrak{a}\) of \(\mathfrak{O}\) is a subset such that for each \(r,s\in \mathfrak{a}\) and \(a\in \mathfrak{O}\) we have closure under addition \(r+s\in \mathfrak{a}\) and \(ra, ar\in \mathfrak{a}\).
For two ideals \(\mathfrak{a}\) and \(\mathfrak{b}\) of \(\mathfrak{O}\), define the sum and product by
  \begin{align*}
   \mathfrak{a}+\mathfrak{b} &:= \{a+b: a\in \mathfrak{a}, b\in \mathfrak{b}\},\quad \text{and} \\
%   \end{equation*}
%   \begin{equation*}
   \mathfrak{ab} &:=\biggl\{ \sum_{j=1}^k a_jb_j: a_j\in \mathfrak{a}, b_j\in \mathfrak{b} \biggl\}.
  \end{align*}
Two ideals $\mathfrak{a}$ and $\mathfrak{b}$ are said to be \emph{coprime} if $\mathfrak{a}+\mathfrak{b}=\mathfrak{O}$.
In this case we have $\mathfrak{a}\mathfrak{b}=\mathfrak{a}\cap\mathfrak{b}$.
%
%
 %For each pair of ideals \(\mathfrak{a},\mathfrak{b}\subset \mathfrak{O}\) we have the inclusions \(\mathfrak{a}\mathfrak{b} \subset \mathfrak{a}\cap\mathfrak{b} \subset \mathfrak{a} \subset \mathfrak{a}+\mathfrak{b}\) (and similarly for \(\mathfrak{b}\)).  In addition, if \(\mathfrak{a}+\mathfrak{b}=\mathfrak{O}\), in which case the ideals are said to be \textit{co-prime}, then \(\mathfrak{a}\mathfrak{b}=\mathfrak{a}\cap\mathfrak{b}\). 
  %We assume, unless otherwise stated, that an ideal \(\mathfrak{a}\) is a proper ideal of \(\mathfrak{O}\).

For ideals \(\mathfrak{a},\mathfrak{b}\) of \(\mathfrak{O}\) we say \(\mathfrak{a}\) \textit{divides} \(\mathfrak{b}\), 
and write \(\mathfrak{a}\vert\mathfrak{b}\), if there exists an ideal \(\mathfrak{c}\) such that \(\mathfrak{b}=\mathfrak{a}\mathfrak{c}\). 
In addition, \(\mathfrak{a}\vert \mathfrak{b}\) if, and only if, \(\mathfrak{b}\subseteq \mathfrak{a}\).  
An ideal \(\mathfrak{p}\) of \(\mathfrak{O}\) is called \textit{prime} if \(\mathfrak{p}\vert \mathfrak{ab}\) implies  \(\mathfrak{p}\vert \mathfrak{a}\) or \(\mathfrak{p}\vert \mathfrak{b}\).  
Recall that the ring of integers \(\mathfrak{O}\) of a number field does not necessarily have unique factorisation.  
%Unlike the ring of rational integers \(\mathbb{Z}\), the elements of \(\mathfrak{O}\) do not necessarily factor uniquely as a product of prime elements.
%
However every non-zero ideal of \(\mathfrak{O}\) can be written as a product of prime ideals and, in addition, this factorisation is unique up to the order of the factors.  

Let \(\mathfrak{a}\) be a non-zero ideal of \(\mathfrak{O}\) then the quotient ring \(\mathfrak{O}/\mathfrak{a}\) is finite, which leads us to define the \textit{norm} of \(\mathfrak{a}\) by \(N(\mathfrak{a})= |\mathfrak{O}/\mathfrak{a}|\). This norm has a multiplicative property: \(N(\mathfrak{ab})=N(\mathfrak{a})N(\mathfrak{b})\) for every pair of non-zero ideals \(\mathfrak{a},\mathfrak{b}\) of \(\mathfrak{O}\). We can connect norms of elements and ideals as follows.  Suppose that \(a\in\mathfrak{O}\) is non-zero then \(N(a\mathfrak{O})=|N_K(a)|\) and, in addition, if \(a\in\mathbb{Q}\)  then \(N(a\mathfrak{O})=|a^{d}|\) where \(d=[K\colon\mathbb{Q}]\).

Suppose that \(\mathfrak{p}\) is a prime ideal. Since the quotient ring \(\mathfrak{O}/\mathfrak{p}\) is a finite field and, by definition, \(N(\mathfrak{p})=|\mathfrak{O}/\mathfrak{p}|\), we conclude that \(N(\mathfrak{\mathfrak{p}})=p^f\) where \(f\le [K\colon\mathbb{Q}] \) and \(p\) is a rational prime. %In general, \(\mathfrak{a}\vert N(\mathfrak{a})\mathfrak{O}\) for any non-zero ideal \(\mathfrak{a}\) of \(\mathfrak{O}\). 
Indeed, \(p\in\mathfrak{p}\) and, further, it is the only rational prime in \(\mathfrak{p}\). Thus, we say that the prime ideal \(\mathfrak{p}\) \textit{lies above} the prime ideal \(p\Z\).  We will frequently use the following version of Fermat's Little Theorem:
 \begin{theorem} \label{thm: FLTField}
 For any prime ideal $\mathfrak{p}$ and algebraic integer \(\lambda\in\mathfrak{O}\), \(\lambda^{N(\mathfrak{p})} - \lambda \in
 \mathfrak{p}\).
 \end{theorem}

%
%
%
%\subsection*{Linear Recurrence Sequences}

We now recall some of the terminology connecting linear recurrence
sequences and exponential polynomials. For further details on this correspondence we refer the reader to ~\cite{everest2003recurrence}.

We call a sequence of algebraic numbers \((u_n)_{n=0}^\infty\) satisfying a recurrence relation \(u_n = a_1u_{n-1} + a_2 u_{n-2} + \cdots + a_\ell u_{n-\ell}\) with fixed real algebraic constants \(a_1,\ldots, a_\ell\) such that \(a_\ell\neq 0\) a \textit{linear recurrence sequence}.  Together with the recurrence relation, the sequence is wholly determined by the initial values \(u_0,\ldots, u_{\ell-1}\).  The polynomial \(f(x) = x^\ell -a_1x^{\ell-1} -\cdots -a_{\ell-1}x - a_\ell\) is called the \textit{characteristic polynomial} associated to the relation.  %Moreover, this relation is said to have \textit{order} equal to the degree of \(f\). 
Associated to each linear recurrence sequence \((u_n)\) is a recurrence relation of minimal length. We call the characteristic polynomial of this minimal length relation the \textit{minimal polynomial} of the sequence.  Moreover, given a recurrence relation the minimal polynomial divides any characteristic polynomial.  The \textit{order} of a linear recurrence sequence is the degree of its minimal polynomial.

 Let \(\mu\) be the minimal polynomial of a linear recurrence sequence \((u_n)\) and \(K\) the splitting field of \(\mu\).  Over \(K\) the polynomial factorises as a product of powers of distinct linear factors \(\mu(x) = \prod_{i=1}^m (x-\lambda_i)^{n_i}\).
Here the constants \(\lambda_1,\ldots, \lambda_m\in K\) are the
\textit{characteristic roots} of \((u_n)\) with multiplicities
\(n_1,\ldots, n_m\). The terms of a linear recurrence sequence can be
realised as an \textit{exponential polynomial} such that \(u_n=
\sum_{i=1}^m A_i(n)\lambda_i^n\). Here the \(\lambda_i\) are the
distinct characteristic roots of the recurrence \((u_n)\) alongside
polynomial coefficients \(A_k\in K[x]\). If the characteristic
polynomial of a sequence has no repeated roots, the terms in the
sequence are each given by an exponential polynomial  \(u_n=
\sum_{i=1}^m A_i(0)\lambda_i^n\) with constant coefficients.  A linear recurrence sequence that satisfies this condition is called \textit{simple}.

%We restrict our consideration to linear recurrence sequences that are %not identically zero, in which case the minimal polynomial of such a %sequence is separable. Under this assumption the number field \(K\) is %Galois over \(\mathbb{Q}\).  

Suppose that \((u_n)_{n=0}^\infty\) is a linear recurrence sequence with characteristic roots \(\lambda_1,\ldots, \lambda_m\in K\).  
For each \(i\in\{1,\ldots, m\}\) there exist  non-zero \(q_i\in\Z\) such that \(q_i\lambda_i\in \mathfrak{O}\).  
Consider the linear recurrence sequence \((w_n)_{n=0}^\infty\) with terms given by \(w_n = q_1^n\cdots q_m^n u_n\).  
By construction, \(w_n = 0\) if and only if \(u_n=0\) and, further, the characteristic roots of \((w_n)\) are algebraic integers in \(\mathfrak{O}\).   Thus, without loss of generality, we assume that each \(\lambda_i\in\mathfrak{O}\) and, in addition, that \(A_1,\ldots,A_m\in\mathfrak{O}[x]\).

Let \((u_n)\) be a linear recurrence sequence with terms \(u_n = A_1(n)\lambda_1^n + \cdots + A_m(n)\lambda_m^n\)
%    \begin{equation*}
%        u_n = A_1(n)\lambda_1^n + \cdots + A_m(n)\lambda_m^n
%    \end{equation*}
where \(\lambda_1,\ldots, \lambda_m\in \mathfrak{O}\) and \(A_1,\ldots, A_m\in\mathfrak{O}[x]\).  We associate to \((u_n)\) a simple linear recurrence \((v_n)\) given by an exponential polynomial \(v_n = A_1(0)\lambda_1^n + \cdots + A_m(0)\lambda_m^n\). 

We are interested in determining whether \(u_n=0\) for \(n=\ell p^k\) with \(k,\ell\in\N\) bounded and \(p\) any rational prime.  In particular, our method is limited to those coefficients \(\ell \in\{0,1,\ldots, c\}\) for which \(v_\ell\neq 0\).  We introduce the set \(\mathcal{L}_c=\{\ell\in\N : \ell\le c,\; v_\ell\neq 0\}\) consisting of such coefficients.  In the case that \((u_n)_{n=0}^\infty\) is simple we have that \(u_n=v_n\) for each \(n\in\N\), and so we need only consider the \(\ell\le c\) such that \(u_\ell\neq 0\).  In the case that \((u_n)_{n=0}^\infty\) is not simple it is possible that \((v_n)\) is identically zero; for example, \(u_n = n\lambda^n\).  
If \(v_0\neq 0\) then \((v_n)\) is not identically zero.
Otherwise \(v_0=u_0=0\) and we have identified a zero term at an index of the desired form.  
%We can ensure that \(v_0\neq 0\) by shifting the linear recurrence sequence \((u_n)_{n=0}^\infty\) a finite number of times (at most the order of the sequence).
%We also note that ultimately recurrent sequences do not satisfy a linear recurrence relation of the form outlined above.  Thus \((v_n)\) does not have a tail of zeroes.

\section{Coefficients in \texorpdfstring{$\mathbb{Z}[x]$}{Z[x]}}

\subsection{Decidability results}
Given a positive rational integer \(n\), recall the multinomial expansion with exponent \(n\) is given by the identity
	\begin{equation*}
		(A_1 x_1 + \cdots + A_m x_m)^n = \sum_{b_1 + \cdots + b_m = n} \binom{n}{b_1,b_2,b_3,\ldots,b_m} \prod_{t=1}^m A_{t}^{b_t} x_t^{b_t}
	\end{equation*}
with the combinatorial coefficient representing the quotient
	\begin{equation*} \label{eq: multicoeff}
		\binom{n}{b_1,b_2,b_3,\ldots,b_m} = \frac{n!}{b_1! b_2! \cdots b_m!}.
	\end{equation*}
%\end{lem}
%
%

We shall make use of the following result, commonly called the \textit{freshman's dream}.
\begin{cor} \label{cor: freshman}
Suppose that \(A_1,\ldots, A_m\in\Z\) and \(\lambda_1,\ldots, \lambda_m\) lie in the ring \(\mathfrak{O}\) of integers of 
some number field \(k\).  Then for any prime \(p\) and \(k\in\N\) we have the following congruence:
	\begin{equation*}
		(A_1 \lambda_1 + \cdots + A_m \lambda_m)^{p^k} \equiv A_1 \lambda_1^{p^k} + \cdots + A_m \lambda_m^{p^k} \pmod{p\mathfrak{O}}.
	\end{equation*}
\end{cor}
\begin{proof}
Let us expand the left-hand side using the aforementioned multinomial identity. 
Now consider each of the combinatorial coefficients in this expansion.
 If exactly one of the choices $b_1,\ldots,b_t$ is equal to $p^k$
then the corresponding coefficient is equal to $1$, and otherwise 
it is an integer multiple of $p$.
Hence
 	\begin{equation*}
		(A_1 \lambda_1 + \cdots + A_m \lambda_m)^{p^k} \equiv A_1^{p^k} \lambda_1^{p^k} + \cdots + A_m^{p^k} \lambda_m^{p^k} \pmod{p\mathfrak{O}}.
	\end{equation*}
The result follows by repeated application of Fermat's Little Theorem, \(A_i^{p^k} \equiv A_i \pmod{p\Z}\). % to each of the terms on the right-hand side.
\end{proof}

In combination with Corollary~\ref{cor: freshman}, we use the following technical lemma in the proof of  Theorem~\ref{thm: FirstCase1}.
\begin{lem} \label{lem: pbound} Suppose that \(b\in\mathfrak{O}\) is non-zero. There are only finitely many rational primes \(p\) such that \(p\mathfrak{O} \vert b\mathfrak{O}\) and, in addition, \(N(b\mathfrak{O})\) is an effective bound on such primes.
\end{lem}
\begin{proof}
Since the ideal norm is multiplicative we have \(p^d = N(p\mathfrak{O})\vert N(b\mathfrak{O})\) where \(d=[K \colon\mathbb{Q}]\).  We can calculate \(N(b\mathfrak{O})\in\Z\) and so obtain an effective bound on any rational prime \(p\) such that \(p\mathfrak{O} \vert b\mathfrak{O}\).
\end{proof}

\begin{proof}[Proof of Theorem~\ref{thm: FirstCase1}]
Let us assume that the algebraic integers \(\lambda_1,\ldots,
\lambda_m\) all lie in a given number field $K$, and let us denote by
$\mathfrak{O}$ the ring of algebraic integers in $K$.  
We note that  it is decidable whether \(u_{p^0}=u_1=A_1 + \cdots + A_m = 0\).  
Thus we can assume, without loss of generality, that \(u_1\neq 0\).
We shall prove the case \(k=1\).  
The proof for higher powers follows with only minor changes to the argument below.

By Corollary~\ref{cor: freshman}, the following congruence holds modulo \(p\mathfrak{O}\),
	\begin{equation*}
		u_{1}^p = (A_1 \lambda_1 + \cdots + A_m \lambda_m)^p \equiv A_1\lambda_1^{p} + \cdots + A_m\lambda_m^{p} = u_{p}.
	\end{equation*}
Thus  \(u_1^{p}\) and \(u_{p}\) lie in the same coset of \(p\mathfrak{O}\).  It follows that \(u_{p}=0\) only if \(u_1^{p}\in p\mathfrak{O}\). 
Since \(p\mathfrak{O} \vert u_1^p\mathfrak{O}\) and \(u_1\neq 0\) (by assumption), we can apply Lemma~\ref{lem: pbound}. 
As \(N(u_1^p\mathfrak{O})\) has only finitely many prime divisors, we obtain an effective bound on
the rational primes \(p\) such that \(u_{p}=0\).   
We have the desired result:  given
\(c\in\N\), it is decidable whether there exists an \(n\in\{p : p
\text{ prime}\}\) such that \(u_n=0\). %
%This concludes the proof sketch of Theorem~\ref{thm: FirstCase1}.
%
\end{proof}
%%a \textit{norm} for ideals; that
%%is, to each non-zero ideal \(\mathfrak{a}\) of \(\mathfrak{O}\) we
%%associate a non-negative integer
%%\(N(\mathfrak{a})=|\mathfrak{O}/\mathfrak{a}|\).  This norm is
%%computable and possesses a multiplicative property.  

We now turn our attention to decidability results for 
linear recurrence sequences whose terms are given by an exponential polynomial with polynomial coefficients in $\mathbb{Z}[x]$.

Let \((u_n)\) be a linear recurrence sequence whose terms are given by \(u_n = A_1(n)\lambda_1^n + \cdots + A_m(n)\lambda_m^n\) with \(A_1,\ldots, A_m\in \mathbb{Z}[x]\) and \(\lambda_1,\ldots, \lambda_m\in\mathfrak{O}\) for some ring of integers in a number field \(K\).
We associate a simple sequence \((v_n)\) with terms given by \(v_n = A_1(0)\lambda_1^n + \cdots + A_m(0)\lambda_m^n\) to each such sequence \((u_n)\).
Given \(c\in\N\), we define the set \(\mathcal{N}_c \subset \N\) as follows:
  \begin{equation*}
   \mathcal{N}_c := \bigcup_{\ell\in\mathcal{L}_c} \{\ell p^k : p \text{ prime},\; k\le c\}.
  \end{equation*}
We recall the set \(\mathcal{L}_c = \{\ell\in\N : \ell\le c,\; v_\ell\neq 0\}\) defined in the previous section.  
Hence \(\mathcal{N}_c\) implicitly depends on the sequence \((u_n)\).
If \(u_0=0\) then we have identified a zero term at a desired index. 
 Otherwise \(u_0\neq 0\) and so, for \(c\) sufficiently large, \(\mathcal{N}_c\) is infinite. 
The goal of this section is to prove the following theorem.
\begin{theorem} \label{thm: SecondCase} 
%Let \((u_n)\) be a linear recurrence sequence whose terms are given by \(u_n = A_1(n)\lambda_1^n + \cdots + A_m(n)\lambda_m^n\) with \(A_1,\ldots, A_m\in \mathbb{Z}[x]\) and \(\lambda_1,\ldots, \lambda_m\in\mathfrak{O}\).
Let \((u_n)\) be a linear recurrence sequence whose terms are given by an exponential polynomial with rational polynomial coefficients as above. 
Fix \(c\in\N\).  
Then one can decide whether there is an \(n\in \mathcal{N}_c\) such that \(u_n = 0\).
\end{theorem}
%
%%
%%%
%%%%
%%%%%

%In order to prove Theorem~\ref{thm: SecondCase}, we use the next technical lemma. 
 Lemma~\ref{lem: modp2} below is a generalisation of Corollary~\ref{cor: freshman} in two senses: the lemma considers sequences that are not necessarily simple and indices of the form \(\ell p^k\in\N\). % We show that \(v_\ell^{p^k}\) and \(u_{\ell p^k}\) belong to the same coset of \(p\mathfrak{O}\) in \(\mathfrak{O}\).
%
%%
%%%
%%%
%
%%
%%

\begin{lem} \label{lem: modp2} Let \((u_n)\) be a recurrence sequence as above and \((v_n)\) the associated simple recurrence sequence.  Let \(p\in\N\) be prime and \(k,\ell\in\N\).  Then \(v_\ell^{p^k} - u_{\ell p^k} \in p\mathfrak{O}\).
\end{lem}
\begin{proof}
We prove the case when \(k=1\).  The general case, dealing with higher powers \(p^k\), follows with only minor changes.  

First, we have the congruence \(v_\ell^p \equiv v_{\ell p} \pmod{p\mathfrak{O}}\) by Corollary~\ref{cor: freshman} since
%  \begin{equation*}
%   v_l^p \equiv A_1(0) \lambda_1^{lp} + \cdots + A_m(0) \lambda_m^{lp} \equiv u_{lp}
%  \end{equation*}
  \begin{equation*}
   \mleft( A_1(0)\lambda_1^\ell + \cdots + A_m(0)\lambda_m^\ell \mright)^p 
%            &\equiv A_1(0)^p\lambda_1^{\ell p} + \cdots + A_m(0)^p\lambda_m^{\ell p} \\
\equiv A_1(0)\lambda_1^{\ell p} + \cdots + A_m(0)\lambda_m^{\ell p}.
  \end{equation*}
%
%%%
%%The first congruence follows by .  The second congruence is an application of Fermat's Little Theorem: \(A_i(0)^p \equiv A_i(0) \pmod{p\Z}\).  

Recall that for \(A\in\mathbb{Z}[x]\) we have \((x-y)\vert (A(x)-A(y))\).  By induction, one can show that \(p \vert (A(lp) - A(0))\) and so \(A(0) \equiv A(\ell p) \pmod{p\Z}\) for each \(A\in\mathbb{Z}[x]\).  This is sufficient to deduce a second congruence
	\begin{equation*}
		v_{\ell p} \equiv A_1(\ell p)\lambda_1^{\ell p} + \cdots + A_m(\ell p)\lambda_m^{\ell p} = u_{\ell p} \pmod{p\mathfrak{O}}.
	\end{equation*}
Together these two congruences give \(v_\ell^p - u_{\ell p} \in p\mathfrak{O}\), the desired result.
\end{proof}

\begin{proof}[Proof of Theorem~\ref{thm: SecondCase}] 
Let us consider the case that \(k=1\).
As previously noted, we can assume there is an \(\ell\le c\) and \(v_\ell\neq 0\) (otherwise \(u_0=0\)). 
Suppose that \(u_{\ell p}=0\). 
Then, by Lemma~\ref{lem: modp2}, \(v_\ell^p \in p\mathfrak{O}\) and so \(p\mathfrak{O} \vert v_\ell^p \mathfrak{O}\). 
Thus \(p \vert N(v_\ell^p \mathfrak{O})\).
Since \(\mathfrak{O}\) is a commutative ring and the ideal norm is multiplicative, we have that \(p \vert N(v_\ell\mathfrak{O})\).
%\(v_\ell^p \mathfrak{O}=(v_\ell\mathfrak{O})^p\). 
By Lemma~\ref{lem: pbound}, we obtain an effective bound on the divisors of \(v_\ell\mathfrak{O}\)  of the form \(p\mathfrak{O}\) and hence a bound on the rational primes for which \(u_{\ell p} = 0\) is possible.
%
%  From the multiplicative property of the norm and the fact that \(p\) is a rational prime, \(p\vert N(v_\ell \mathfrak{O})\).  Since there are only finitely many prime divisors of \(N(v_\ell \mathfrak{O})\), we have an effective bound on the rational primes that permit \(u_{\ell p} =0\). 
Mutatis mutandis the proof holds for prime powers \(p^k\) with \(k> 1\).  Clearly the case \(k=0\) is decided by determining whether \(u_{\ell}=0\).
\end{proof}
%
%
%

%%%%%
%%%%%
%%%%%

\subsection{Complexity upper bound}

Given a simple linear recurrence sequence \((u_n)\), we establish a quantitative bound on the magnitude of any prime $p$ such that \(u_p=0\).
The bound is
in terms of the size of the problem instance.  
In the case that \((u_n)\) is a simple linear recurrence sequence, we know that \(u_n = A_1\lambda_1^n + \cdots + A_m\lambda_m^n\) and so the size of the problem instance is the bit length \( S = \| \langle \lambda_1, \lambda_2, \ldots, \lambda_m, A_1, A_2,\ldots, A_m \rangle \|\).

We give the following rudimentary bounds in terms of \(S\).  First, we bound \(\log_2 |A_i| +1\), bit length of the integer \(A_i\), from above by \(2^S\).  Second, \(|\lambda_i|\) is bounded from above by \(H(\lambda_i) \le 2^S\) where the height \(H(\lambda_i)\) is the maximum absolute value of the coefficients in \(\mu_{\lambda_i}\).  Finally, we have \(\deg(\lambda_i) \le S\), from which it follows that \([K\colon \mathbb{Q}] = [\mathbb{Q}(\lambda_1,\ldots, \lambda_m) \colon \mathbb{Q}] \le m^S \le S^S\).  Because \(u_1 = A_1\lambda_1 + \cdots A_m \lambda_m\) we have the following elementary bound
    \begin{equation*}
    %\mleft| \prod_{\ell =1}^{[K \colon \mathbb{Q}]} \sigma_\ell (u_1) \mright|
        N(u_1 \mathfrak{O}) \le \prod_{\ell =1}^{[K \colon \mathbb{Q}]} \sum_{k=1}^m \mleft|\sigma_\ell(A_k)\sigma_\ell(\lambda_k) \mright| \le \prod_{\ell =1}^{[K \colon \mathbb{Q}]} S2^{3S} \le \bigl(S2^{3S}\bigr)^{S^S}.
    \end{equation*}

From the above calculations it follows that if $u_p=0$ for some prime $p$ then $p$ is at most $(S2^{3S})^{S^S}$, i.e., double exponential in $S$, the size of the problem instance.

%It is clear that the limiting factor in our complexity upper bound comes from testing prime factors of an integer whose magnitude is given by a double exponential bound in \(S\).

%
%
%
%
%

\section{Coefficients in \texorpdfstring{$\mathfrak{O}[x]$}{O[x]}}

Let us first recall some background material on the decomposition of prime ideals in the ring of integers \(\mathfrak{O}\) 
of a Galois number field \(K\). %
Such decompositions (as products of powers of prime ideals) are particularly well-behaved in this setting---
a comprehensive presentation of this material can be found in~\cite{cohen1993computational}.
Let \(p\in\N\) be prime.
Then \(p\mathfrak{O} = \prod_{i=1}^g \mathfrak{p}_i^{e}\) where the \(\mathfrak{p}_i\) are the prime ideals lying above \(p\mathbb{Z}\). 
Here the integer \(e(p)\ge 1\) is the \textit{ramification index} of \(p\).
The degree of the field extension \(f(p) =[\mathfrak{O}/\mathfrak{p}_i : \Z/p\Z]\), the \textit{inertial degree} of \(\mathfrak{p}_i\) over \(p\Z\), is independent of the prime ideal \(\mathfrak{p}_i\).   
Suppose that \(\mathfrak{p}\) lies above \(p\Z\). 
We have \(N(\mathfrak{p}) = N(p\Z)^{f(p)} = p^{f(p)}\). 
A prime \(p\mathbb{Z}\) is \textit{ramified} in \(\mathfrak{O}\) if \(e>1\) and \textit{unramified} otherwise.  
In particular, only finitely many primes ramify in \(\mathfrak{O}\) since \(p\mathbb{Z}\) ramifies in \(\mathfrak{O}\) if, and only if, \(p\) divides the discriminant of \(K\) (see e.g. \cite{cohen1993computational}).
%
%
%
%
%%
%%%

Suppose that \(K\) is Galois over \(\mathbb{Q}\) and let \(\mathfrak{O}\) be the algebraic integers in \(K\).
In this section we shall prove
decidability results locating the zeroes of sequences \((u_n)\) whose
terms are given by an exponential polynomial of the form \(u_n = A_1(n)\lambda_1^n + \cdots + A_m(n)\lambda_m^n\) with coefficients \(A_1,\ldots, A_m\in \mathfrak{O}[x]\) and \(\lambda_1,\ldots, \lambda_m\in\mathfrak{O}\). 
For such a sequence, fix \(c\in\N\) and let \(\mathcal{L}_c=\{\ell\in\N : \ell\le c,\; v_\ell\neq 0\}\) where \((v_n)\) is the simple recurrence sequence with terms given by \(v_n = A_1(0)\lambda_1^n + \cdots + A_m(0)\lambda_m^n\). 
Let \(f(p)\) be the inertial degree of \(p\Z\) in \(\mathfrak{O}\). 
 Then define the set \(\mathcal{N}_c(K)\) as the union
  \begin{equation*}
    \mathcal{N}_c(K) = \bigcup_{\ell\in\mathcal{L}_c} \{\ell p^{kf(p)} : p \text{ prime},\; k\le c\}.
  \end{equation*}
Here our choice of notation is meant to draw comparison with our previous definition for the set \(\mathcal{N}_c\). 
Without loss of generality we assume that given \(c\in\N\) there is an \(l\le c\) such that \(v_\ell \neq 0\) for otherwise the sequence \((u_n)\) vanishes at \(u_0 = v_0 = 0\).
%; however, we do not restrict \(k\le c\) in the above definition (in this respect the results in the next section are marginally stronger than the previous section).
We denote by \(\mathcal{Q}_c(K)\) the subset 
  \begin{equation*}
   \mathcal{Q}_c(K) = \bigcup_{\ell \in\mathcal{L}_c} \{\ell p^{kf(p)} : p\Z \text{ unramified},\; k\le c\}.
  \end{equation*}
Similarly, let \(\mathcal{R}_c(K)\subset \mathcal{N}_c(K)\) be the corresponding set of elements where \(p\mathbb{Z}\) is ramified in \(\mathfrak{O}\). Since there are only finitely many prime ideals \(p\Z\) that are ramified in \(\mathfrak{O}\), the cardinality of the set \(\mathcal{R}_c(K)\) is finite.  By definition, \(\mathcal{N}_c(K) = \mathcal{Q}_c(K)\cup \mathcal{R}_c(K)\).

Our main result is the following theorem.
\begin{theorem} \label{thm: GeneralResult} Fix \(c\in\N\).  Given \((u_n)\) as above, one
  can decide whether there is an \(n\in \mathcal{N}_c(K)\) such that \(u_n=0\).
\end{theorem}
Since the set \(\mathcal{R}_c(K)\) is finite, locating zero terms \(u_n=0\) for \(n\in\mathcal{R}_c(K)\) is clearly decidable.  So to prove Theorem~\ref{thm: GeneralResult} it is sufficient to prove the next theorem.

\begin{theorem} \label{thm: UnramifiedResult} Fix \(c\in\N\).  Given \((u_n)\) as above,
  one can decide whether there is an \(n\in \mathcal{Q}_c(K)\) such that \(u_n=0\).
\end{theorem}

In order to prove Theorem~\ref{thm: UnramifiedResult}, we first prove two technical results. The first, Lemma~\ref{lem: FLTgeneral}, concerns elements of cosets of \(p\mathfrak{O}\) in \(\mathfrak{O}\).  The second, Lemma~\ref{lem: modp3}, plays an analogous r\^{o}le to that of Lemma~\ref{lem: modp2} in Section 3.
\begin{lem} \label{lem: FLTgeneral}  Suppose that \(\varphi\in\mathfrak{O}\) and \(p\Z\) is non-zero prime ideal.  If \(p\Z\) is unramified with inertial degree \(f(p)\) then \(\varphi^{p^{f(p)}} - \varphi \in p\mathfrak{O}\).
\end{lem}
\begin{proof}
Write \(p\mathfrak{O}=\mathfrak{p}_1\cdots \mathfrak{p}_g\) for the unique factorisation of \(p\mathfrak{O}\) as a product of the distinct prime ideals \(\mathfrak{p}_i\) lying above \(p\Z\).  
Here the ramification index is unity because \(p\Z\) is unramified.
By Theorem~\ref{thm: FLTField}, for each \(i\in\{1,\ldots, g\}\) and \(\varphi\in\mathfrak{O}\) we have  \(\varphi^{N(\mathfrak{p}_i)} - \varphi \in\mathfrak{p}_i\).  
Since each of the exponents satisfy \(N(\mathfrak{p}_i)=p^{f(p)}\), we deduce that \(\varphi^{p^{f(p)}} - \varphi \in \cap_i \mathfrak{p}_i\). 
Because the distinct prime ideals \(\mathfrak{p}_i\) are pairwise co-prime, we have \(\cap_i \mathfrak{p}_i = \mathfrak{p}_1\cdots \mathfrak{p}_g= p\mathfrak{O}\) and hence we have the desired result.
\end{proof}

\begin{lem} \label{lem: modp3}
Let \((u_n)\) be a recurrence sequence and \((v_n)\) the associated simple recurrence sequence as above.
Let \(p\in\N\) be a rational prime and \(k,\ell\in\N\).
If \(p\Z\subset\mathfrak{O}\) is unramified with inertial degree \(f(p)\) then \(v_\ell - u_{\ell p^{kf(p)}} \in  p\mathfrak{O}\).
\end{lem}
\begin{proof}
The result is a consequence of the next congruences
%  \begin{equation*}
%   v_l \equiv v_l^{p^{kf(p)}} \equiv A_1(0)\lambda_1^{lp^{kf(p)}} + \cdots + A_m(0)\lambda_m^{lp^{kf(p)}} \equiv u_{lp^{kf(p)}}
%  \end{equation*}
  \begin{equation*}
   v_\ell \equiv v_{\ell p^{kf(p)}} \equiv u_{\ell p^{kf(p)}} \pmod{p\mathfrak{O}}.
  \end{equation*}
The congruences hold trivially when \(k=0\).
We shall prove the case \(k=1\) below and omit the case \(k>1\) as it follows similarly.
The first congruence is a simple application of Lemma~\ref{lem: FLTgeneral}:
  \begin{equation*}
  	v_\ell = \sum_{j=1}^m A_j(0)\lambda_j^\ell \equiv \sum_{j=1}^m  A_j(0)\lambda_j^{\ell{{p^{f(p)}}}} = v_{\ell p^{f(p)}} \pmod{p\mathfrak{O}}.
	\end{equation*}
%We also have, by Lemma~\ref{lem: FLTgeneral}, \(A_i(0)^{p^{kf(p)}} \equiv A_i(0) \pmod{p\mathfrak{O}}\) and so
%
Recall that for \(A\in\mathfrak{O}[x]\) we have \((x-y)\vert (A(x)-A(y))\). The second congruence holds since \(p\mathfrak{O} \ni \ell p^{f(p)} \vert (A(\ell p^{f(p)}) -A(0))\) or equivalently \(A(0) \equiv A(\ell p^{f(p)}) \pmod{p\mathfrak{O}}\) for each \(A\in\mathfrak{O}[x]\).
Thus
	\begin{equation*}
		v_{\ell p^{f(p)}} \equiv \sum_{j=1}^m A_j \mleft( \ell p^{f(p)} \mright) \lambda_j^{\ell{{p^{f(p)}}}} = u_{\ell p^{f(p)}} \pmod{p\mathfrak{O}}.
	\end{equation*}
  Hence \(v_\ell - u_{\ell p^{f(p)}} \in p\mathfrak{O}\) as desired.
\end{proof}

\begin{proof}[Proof of Theorem~\ref{thm: UnramifiedResult}]
Fix \(c\in\N\)
and assume that \(n\in \mathcal{Q}_c(K)\) such that \(u_n=0\).
Then \(n\) is of the form \(\ell p^{kf(p)}\) where \(p\) is a prime and \(p\Z\subset \mathfrak{O}\) is unramified. 
By Lemma~\ref{lem: modp3}, \(v_\ell - u_{\ell p^{kf(p)}} \in  p\mathfrak{O}\). 
Thus \(v_\ell \in {p\mathfrak{O}}\) and therefore \(p\mathfrak{O} \vert v_\ell \mathfrak{O}\).   
We then apply Lemma~\ref{lem: pbound} to give an effective bound on the primes by a divisibility argument for \(N(v_\ell \mathfrak{O})\). Hence the result.
%and so have an effective bound on the rational primes \(p\) for which \(u_{\ell p^{kf(p)}}=0\).
\end{proof}
%
%
%

%Let \((u_n)\) be a linear recurrence sequence with terms given by \(u_n = A_1(n)\lambda_1^n + \cdots + A_m(n)\lambda_m^n\) with polynomial coefficients \(A_1,\ldots, A_m\in\mathfrak{O}[x]\). We associate to the sequence \((u_n)\) a simple linear recurrence sequence \((v_n)\) with terms given by \(v_n = A_1(0)\lambda_1^n + \cdots + A_m(0)\lambda_m^n\). 
Our approach in the proof of Theorem~\ref{thm: GeneralResult} extends in the following way:
we can decide whether there exists there is an \(n=\sum_{j=1}^t l_jp^{k_jf(p)}\) such that \(u_n=0\). 
Here the constants \(k_j, l_j\in\N\) are bounded independently of the rational prime \(p\), and \(f(p)\) is the inertial degree of \(p\Z\subset \mathfrak{O}\).
For \(l_1,\ldots, l_t, k_1, \ldots, k_t\in\N\), we define
	\begin{equation*}
		S_m = S_m(l_j; k_j) := \begin{cases}
						 \sum_{j=1}^t l_jm^{k_j f(m)} & \text{if } m \text{ is prime}, \\
						 \sum_{j=1}^t l_j & \text{if } m=1.
				\end{cases}
	\end{equation*}
Fix \(c\in\N\) and, as before, let \(\mathcal{L}_c = \{\ell\in\N \colon \ell\le c,\, v_\ell\neq 0\}\).  Define the set \(\mathcal{N}_c'(K)\) as follows
  \begin{equation*}
   \mathcal{N}_c'(K) = \bigcup_{S_1 \in \mathcal{L}_c} \mleft\{ S_p(l_j; k_j) \colon p\ \textit{prime},\, k_j\le c \mright\}.
  \end{equation*}	
  We define the sets \(\mathcal{Q}_c'(K)\), for unramified \(p\Z\) in \(K\), and \(\mathcal{R}_c'(K)\), for ramified \(p\Z\) in \(K\), in an analogous manner to the sets \(\mathcal{Q}_c(K)\) and \(\mathcal{R}_c(K)\) associated to \(\mathcal{N}_c(K)\).  
  Then, like before, \(\mathcal{N}_c'(K)=\mathcal{Q}_c'(K)\cup \mathcal{R}_c'(K)\) and \(\mathcal{R}_c'(K)\) has finite cardinality.

We have the next decidability result.
\begin{theorem} \label{thm: sumtheorem} Fix \(c\in\N\).  Then, given \((u_n)\) as above, one can decide whether there is an \(n\in\mathcal{N}_c'(K)\) such that \(u_n=0\).
\end{theorem}
The proof of Theorem~\ref{thm: sumtheorem} follows the approach in the proof of Theorem~\ref{thm: GeneralResult}. Since the cardinality of \(\mathcal{R}_c'(K)\) is finite, we need only prove the next theorem in order to prove Theorem~\ref{thm: sumtheorem}.
\begin{theorem} \label{thm: sumtheoremQ} Fix \(c\in\N\).  Then, given \((u_n)\) as above, one can
  decide whether there is an \(n\in\mathcal{Q}_c'(K)\) such that \(u_n=0\).
\end{theorem}

Given its similarities to the proof of Theorem~\ref{thm: UnramifiedResult}, we omit a formal proof of Theorem~\ref{thm: sumtheoremQ}; instead, we outline the key steps in the proof. 
We require the following technical lemma; Lemma~\ref{lem: sump} generalises the result in Lemma~\ref{lem: modp3}.
% in which plays an analogous r\^{o}le to that of Lemma~\ref{lem: modp3} in the proof of Theorem~\ref{thm: UnramifiedResult}.
%
\begin{lem} \label{lem: sump}
Let \((u_n)\) be a recurrence sequence and \((v_n)\) the associated simple recurrence sequence as above.
Let \(p\in\N\) be a rational prime and \(S_p(l_j ; k_j)\) be defined as above.
If \(p\Z\subset\mathfrak{O}\) is unramified then \(u_{S_p} - v_{S_1} \in p\mathfrak{O}\).
\end{lem}
\begin{proof}
We avoid repeating the proof of Lemma~\ref{lem: modp3} by limiting our presentation to the next two observations.  
First, for each polynomial \(A\in\mathfrak{O}[x]\) we have \(A(S_p) - A(0) \in p\mathfrak{O}\) since \(p\mathfrak{O} \ni S_p\) divides \(A(S_p) -A(0)\).  
Second, by repeated application of Lemma~\ref{lem: FLTgeneral}, we have \(\lambda^{S_p} - \lambda^{S_1} \in p\mathfrak{O}\) for \(\lambda \in \mathfrak{O}\). 
From these observations, one can obtain the congruences \(v_{S_1} \equiv v_{S_p} \equiv u_{S_p} \pmod{p\mathfrak{O}}\) and hence the desired result.
%The remainder of the proof of Lemma~\ref{lem: sump} is then comparable to the proof of Lemma~\ref{lem: modp3}.
\end{proof}
%
%%
%%%

We sketch the key steps in the proof of Theorem~\ref{thm: sumtheoremQ}.
\begin{proof}[Proof of Theorem~\ref{thm: sumtheoremQ}]
 Fix \(c\in\N\).
 Assume that \(u_{S_p}=0\) for some \(S_p(l_j;k_j)\in\mathcal{N}_c'(K)\) where \(p\Z\subset \mathfrak{O}\) is an unramified prime.
Note that \(v_{S_1} \neq 0\) since  \(S_p(l_j;k_j)\in\mathcal{N}_c'(K)\).
 Then, by Lemma~\ref{lem: sump}, \(v_{S_1} \in p\mathfrak{O}\) and so \(p\mathfrak{O} \vert v_{S_1}\mathfrak{O}\).  
 By Lemma~\ref{lem: pbound}, \(p\) necessarily divides \(N(v_{S_1}\mathfrak{O})\).  
 Since \(N(v_{S_1}\mathfrak{O})\) is computable, one can derive an effective bound on the rational primes \(p\) such that \(u_{S_p}=0\).
\end{proof}

\section{Hardness result}
In~\cite{blondel2002zero}, Blondel and Portier proved that the Skolem Problem is \NP-hard (see also~\cite{akshay2017complexity}).
In this section we show that the prime variant of the Skolem Problem is likewise \NP-hard.
Following~\cite{akshay2017complexity}, our proof is by reduction from the \textit{Subset Sum Problem}:
given a finite set of integer \(A=\{a_1,\ldots, a_m\}\) and \(b\in\Z\) a target, written in binary,
decide whether there is a subset \(S\subseteq \{1,\ldots, m\}\) such that \(\sum_{k\in S} a_k = b\).  

Let us state two well-known theorems in number theory in order to derive a simple corollary that is fundamental to our proof of~\autoref{thm:subsetprimered}.
\begin{theorem}[Chinese remainder theorem]
Let \(n_1,\ldots, n_m\) be positive integers that are pairwise co-prime. %; that is, \((n_i, n_j) =1\) for each pair \((i,j)\neq (1,1)\),
Then the system of \(m\) equations \(r \equiv a_k \pmod{n_k}\) with each \(a_k\in\Z\) has a unique solution modulo \(N\) where \(N= n_1n_2\cdots n_m\).
\end{theorem}
Dirichlet proved the following theorem on primes in arithmetic progressions.  We use the notation \((m,n)\) to indicate the greatest common divisor of \(m,n\in\Z\).
\begin{theorem}
    Suppose that \(q\) and \(r\) are co-prime positive integers.  Then there are infinitely many primes of the form \(\ell q + r\) with \(\ell\in\N\).
\end{theorem}

The next corollary is immediate.%  The second assertion is a direct application of Dirichlet's theorem.
\begin{cor} \label{cor:dirichlet}
 Let \(p_1, \ldots, p_m\) be a finite set of distinct primes.  Then the system of \(m\) equations \(r \equiv a_k \pmod{p_k}\) with each \(a_k\in\Z\) has a unique solution \(r \in \{0,1,\ldots, P-1\}\) where \(P = p_1 p_2\cdots p_m\).  Additionally, if \((r,P)=1\) then there are infinitely many \(\ell\in\N\) for which \(\ell P + r\) is prime.
\end{cor}

Recall that the \textit{\(n\)th cyclotomic polynomial} given by
    \begin{equation*}
       \Phi_n(x) = \prod_{\substack{k\in\{1,\ldots, n\} \\ (k,n)=1}} \left(x - \eu^{2\pi \iu k/ n}\right)
    \end{equation*}
is the minimal polynomial over $\mathbb{Q}$ of a primitive $n$th
root of unity.

We call an integer linear recurrence sequence \textit{cyclotomic} if its characteristic roots are all roots of unity.
The next theorem, concerning Skolem's Problem in the restricted setting of cyclotomic sequences, follows from work in~\cite{akshay2017complexity}. We reproduce the proof as a lead into our original work on the Skolem Problem restricted to prime numbers.

\begin{theorem} \label{thm:subsetred}
The cyclotomic Skolem Problem is \NP-hard.
\end{theorem}

The proof of \autoref{thm:subsetred} is by reduction from the \text{Subset Sum} Problem and follows directly from the technical lemma, Lemma~\ref{lem:subsetred}, below.  
Before we present the proof, we introduce some notation.

Let \(\{p_1,\ldots, p_m\}\) be the set of the first \(m\) prime numbers.  
We define the linear recurrence sequence \((s_k(n))_{n=0}^\infty\) with \(k\in \{1,\ldots, m\}\) as follows.  
Let \(s_k(n) = s_k(n-p_k)\) for \(n\ge p_k\) with initial conditions \(s_k(0)=1\), \(s_k(1)=\cdots = s_k(p_k -1) = 0\).  
Then each sequence \((s_k(n))\) is periodic with period \(p_k\).  The characteristic polynomial associated to \((s_k(n))\) is given by
    \begin{equation*}
        x^{p_k} - 1 = \prod_{\ell=0}^{p_k-1} \mleft( x - \eu^{2\pi \iu \ell/p_k} \mright).
    \end{equation*}
Thus \((s_k(n))\) is a cyclotomic sequence.

In order to reduce the Subset Sum Problem to the cyclotomic Skolem Problem, we consider the inhomogeneous linear recurrence sequence \((t(n))_{n=0}^\infty\) with terms given by \(t(n) = b - \sum_{k=1}^m a_k s_k(n)\). 
The characteristic polynomial associated to \((t(n))\) is given by the least common multiple of %%%
    \begin{equation*}
        (x^{p_1}-1)(x-1), x^{p_2}-1, \ldots, x^{p_m}-1
    \end{equation*}%(x^{p_1}-1)(x-1), x^{p_2}-1, \ldots, x^{p_m}-1\) 
\text{(see~\cite{everest2003recurrence})}, from which it follows that each of the characteristic roots of \((t(n))\) are themselves roots of unity, i.e., $(t(n))$ is a cyclotomic sequence.

\begin{lem} \label{lem:subsetred}
For \((t(n))\) given as above, there exists \(N\in\N\) such that \(t(N) = 0\) if and only if the Subset Sum Problem with inputs \(\{a_1,\ldots, a_m; b\}\) has a solution.
\end{lem}
\begin{proof}
Suppose that there exists an \(N\in\N\) such that \(t(N)=0\), then the Subset Sum Problem has a solution because the selectors \(s_k(n)\) are \(\{0,1\}\)-valued.  
Conversely, suppose that there is a subset \(S\subseteq \{1,\ldots, m\}\) such that  \(\sum_{k\in S} a_k =b\) and define \(N = \prod_{k\in S} p_k\).  
We have  \(s_k(N) = 1\) for each \(k\in S\) since \(p_k \mid N\), and \(s_k(N)=0\) otherwise.  
Thus
    \begin{equation*}
        t(N) = b - \sum_{k=1}^m a_k s_k(N) = b - \sum_{k \in S} a_k = 0,
    \end{equation*}
as required.
\end{proof}

We prove the following complexity result for the Skolem Problem for primes.  

\begin{theorem} \label{thm:subsetprimered}  Suppose that \((u_n)\) is a cyclotomic integer linear recurrence sequence.  The problem of deciding whether there is a prime \(p\in\N\) such that \(u_p = 0\) %reduces from the Subset Sum Problem.
is \NP-hard.
\end{theorem}

The proof of \autoref{thm:subsetprimered} involves an analysis of the 
\NP-hardness proof for Skolem's Problem.  Technically we will derive the result from Lemma~\ref{lem:subsetprimered}, below.

Let \(p_1,\ldots, p_m\) be the first \(m\) odd primes.  We define selector sequences \((\sigma_k(n))\) with \(k\in\{1,\ldots, m\}\) as follows.  Let \(\sigma_k(n) = \sigma_k(n-p_k)\) for \(n\ge p_k\) with initial conditions \(\sigma_k(1)=1\), \(\sigma_k(0)= \sigma_k(2) = \cdots = \sigma_k(p_k -1) = 0\).  %Note the change in the initial conditions from our previous application of selector sequences.  
Then each sequence \((\sigma_k(n))\) is periodic with period \(p_k\).  Let \(\tau(n) = b - \sum_{k=1}^m a_k \sigma_k(n) \). 
It is easily shown that \((\sigma_k(n))\) and \((\tau(n))\) are cyclotomic recurrence sequences.

\begin{lem}
 \label{lem:subsetprimered}
There exists an odd prime \(p\in\N\) such that \(\tau(p)=0\) if and only if there exists a subset \(S\subseteq \{1,\ldots, m\}\) that is a solution to the Subset Sum Problem with inputs \(\{a_1,\ldots, a_m; b\}\). % \(\sum_{k\in S} a_k = b\).
\end{lem}

\begin{proof}
Suppose that there is an odd prime \(p\in\N\) such that \(\tau(p)=0\).  Then there is a solution to the Subset Sum Problem as \(\sigma_k(p)\in \{0,1\}\) for each \(k\).

Conversely, suppose that there a subset \(S\subseteq \{1,\ldots, m\}\) such that \(\sum_{k\in S} a_k = b\). Consider the set \(Q(S)\subseteq \Z\) of integer solutions to the set of \(m\) equations
    \begin{equation*}
    \begin{cases}
    r \equiv 1 \pmod{p_k} & \text{if\ } k\in S, \quad \text{and} \\
    r \equiv 2 \pmod{p_k} & \text{if\ } k\in \{1,\ldots, m\}\setminus S.
    \end{cases}
    \end{equation*}
The choice of residue ensures that \(r\) is not divisible by any of the primes \(p_1,p_2,\ldots, p_m\).  By the Chinese Remainder Theorem, \(Q(S)\) is an infinite arithmetic progression.  Suppose that \(q\in Q(S)\). Then, by definition of the selector sequences, \(\sigma_k(q)=1\) if and only if \(q \equiv 1 \pmod{p_k}\) if and only if \(k\in S\).  Then
    \begin{equation*}
        \tau(q) = b - \sum_{k=1}^m a_k \sigma_k(q) = b - \sum_{k\in S} a_k = 0.
    \end{equation*}
    
It remains to show that there is a prime number in \(Q(S)\).  This result follows easily from Corollary~\ref{cor:dirichlet}, which completes the proof.
\end{proof}

\section{Summary}

In this paper we have given decision procedures for finding zeroes of certain prescribed  linear recurrence sequences.
Our main result shows how to decide the existence of a prime \(p\) such that \(u_p=0\) for a simple linear recurrence sequence \((u_n)\). 
We have noted that this decision problem is \NP-hard and, implicitly, that the magnitude of the smallest prime \(p\) such that \(u_p=0\) is at least exponential in the size of the problem instance.
On the other hand, our decision procedure yields a double exponential bound on the magnitude of the prime \(p\).  
Closing this exponential gap would be an interesting direction for further work.
Another direction for research would be to locate zeroes \(u_n=0\) where the index \(n\in\N\) has two prime factors. 

%
%%
%%

%%
%% Bibliography
%%
\bibliographystyle{plainurl}
\bibliography{Skolembib}

\begin{thebibliography}{10}

\bibitem{akshay2017complexity}
S.~Akshay, Nikhil Balaji, and Nikhil Vyas.
\newblock {Complexity of Restricted Variants of Skolem and Related Problems}.
\newblock In K.~Larsen, H.~Bodlaender, and J-F. Raskin, editors, {\em 42nd
  International Symposium on Mathematical Foundations of Computer Science (MFCS
  2017)}, volume~83 of {\em Leibniz International Proceedings in Informatics
  (LIPIcs)}, pages 78:1--78:14, Dagstuhl, Germany, 2017. Schloss
  Dagstuhl--Leibniz-Zentrum fuer Informatik.
\newblock \href {https://doi.org/10.4230/LIPIcs.MFCS.2017.78}
  {\path{doi:10.4230/LIPIcs.MFCS.2017.78}}.

\bibitem{berstel1976deux}
Jean Berstel and Maurice Mignotte.
\newblock Deux propri\'{e}t\'{e}s d\'{e}cidables des suites r\'{e}currentes
  lin\'{e}aires.
\newblock {\em Bulletin de la Soci\'{e}t\'{e} Math\'{e}matique de France},
  104(2):175--184, 1976.

\bibitem{blondel2002zero}
Vincent~D. Blondel and Natacha Portier.
\newblock The presence of a zero in an integer linear recurrent sequence is
  {NP}-hard to decide.
\newblock {\em Linear Algebra Appl.}, 351/352:91--98, 2002.
\newblock Fourth special issue on linear systems and control.
\newblock \href {https://doi.org/10.1016/S0024-3795(01)00466-9}
  {\path{doi:10.1016/S0024-3795(01)00466-9}}.

\bibitem{cohen1993computational}
Henri Cohen.
\newblock {\em A course in computational algebraic number theory}, volume 138
  of {\em Graduate Texts in Mathematics}.
\newblock Springer-Verlag, Berlin, 1993.

\bibitem{derksen2007skolem}
Harm Derksen.
\newblock A {S}kolem-{M}ahler-{L}ech theorem in positive characteristic and
  finite automata.
\newblock {\em Inventiones Mathematicae}, 168(1):175--224, 2007.

\bibitem{everest2003recurrence}
Graham Everest, Alf van~der Poorten, Igor Shparlinski, and Thomas Ward.
\newblock {\em Recurrence sequences}, volume 104 of {\em Mathematical Surveys
  and Monographs}.
\newblock Amer. Math. Soc., Providence, RI, 2003.

\bibitem{halava2005skolem}
Vesa Halava, Tero Harju, Mika Hirvensalo, and Juhani Karhum\"{a}ki.
\newblock {S}kolem's problem--on the border between decidability and
  undecidability.
\newblock Technical report, Turku Centre for Computer Science, 2005.

\bibitem{lech1952recurring}
Christer Lech.
\newblock A note on recurring series.
\newblock {\em Arkiv f\"{o}r Matematik}, 2:417--421, 1953.

\bibitem{mahler1935taylor}
K.~Mahler.
\newblock {E}ine arithmetische {E}igenschaft der {T}aylor-koeffizienten
  rationaler {F}unktionen.
\newblock {\em Proc. Akad. Wet. Amst.}, 38:50--69, 1935.

\bibitem{mahler1956taylor}
K.~Mahler and J.~Cassels.
\newblock On the {T}aylor coefficients of rational functions.
\newblock {\em Mathematical Proceedings of the Cambridge Philosophical
  Society}, 52(1):39–48, 1956.

\bibitem{mignotte1984distance}
Maurice Mignotte, Tarlok Shorey, and Robert Tijdeman.
\newblock The distance between terms of an algebraic recurrence sequence.
\newblock {\em Journal f\"{u}r die Reine und Angewandte Mathematik}, pages
  63--76, 1984.

\bibitem{ouaknine2012decision}
Jo\"{e}l Ouaknine and James Worrell.
\newblock Decision problems for linear recurrence sequences.
\newblock In {\em Reachability problems}, volume 7550 of {\em Lecture Notes in
  Computer Science}, pages 21--28. Springer, Heidelberg, 2012.

\bibitem{ouaknine2015linear}
Jo\"{e}l Ouaknine and James Worrell.
\newblock On linear recurrence sequences and loop termination.
\newblock {\em ACM SIGLOG News}, 2(2):4--13, April 2015.

\bibitem{skolem1934verfahren}
Thoralf Skolem.
\newblock {E}in {V}erfahren zur {B}ehandlung gewisser exponentialer
  {G}leichungen und diophantischer {G}leichungen.
\newblock {\em 8de Skand. Mat. Kongress, Stockholm (1934)}, pages 163--188,
  1934.

\bibitem{vereshchagin1985occurence}
Nikolai Vereshchagin.
\newblock Occurrence of zero in a linear recursive sequence.
\newblock {\em Mathematical notes of the Academy of Sciences of the USSR},
  38(2):609--615, Aug 1985.

\end{thebibliography}

\appendix

\end{document}